\documentclass[11pt]{article}

\usepackage[margin=1.02in]{geometry}
\usepackage{amsmath,amssymb,amsthm,mathtools}
\usepackage{enumitem}
\usepackage{microtype}
\usepackage[T1]{fontenc}
\usepackage{lmodern}
\usepackage{booktabs}
\usepackage{array}
\usepackage{hyperref}
\usepackage[T1]{fontenc}
\usepackage[utf8]{inputenc}
\usepackage{lmodern}
\usepackage{microtype}
\usepackage{amsmath,amssymb,amsthm,mathtools}
\usepackage{mathrsfs}
\usepackage{enumitem}
\usepackage{booktabs}
\usepackage{aliascnt}
\usepackage{hyperref}
\usepackage[nameinlink,capitalise]{cleveref}

\hypersetup{
  colorlinks=true,
  linkcolor=blue,
  citecolor=blue,
  urlcolor=blue,
  pdftitle={Baer Splitting Beyond tau-q-Semisimplicity},
  pdfauthor={Author name(s) to be inserted}
}

\newtheorem{theorem}{Theorem}[section]

\newaliascnt{proposition}{theorem}
\newtheorem{proposition}[proposition]{Proposition}
\aliascntresetthe{proposition}

\newaliascnt{lemma}{theorem}
\newtheorem{lemma}[lemma]{Lemma}
\aliascntresetthe{lemma}

\newaliascnt{corollary}{theorem}
\newtheorem{corollary}[corollary]{Corollary}
\aliascntresetthe{corollary}

\theoremstyle{definition}
\newaliascnt{definition}{theorem}
\newtheorem{definition}[definition]{Definition}
\aliascntresetthe{definition}

\newaliascnt{example}{theorem}

\aliascntresetthe{example}

\theoremstyle{remark}
\newaliascnt{remark}{theorem}
\newtheorem{remark}[remark]{Remark}
\aliascntresetthe{remark}

\newaliascnt{question}{theorem}

\aliascntresetthe{question}

\crefname{theorem}{Theorem}{Theorems}
\crefname{proposition}{Proposition}{Propositions}
\crefname{lemma}{Lemma}{Lemmas}
\crefname{corollary}{Corollary}{Corollaries}
\crefname{definition}{Definition}{Definitions}
\crefname{remark}{Remark}{Remarks}
\crefname{question}{Question}{Questions}
\crefname{example}{Example}{Examples}

\newcommand{\Reg}{\operatorname{Reg}}
\newcommand{\Min}{\operatorname{Min}}
\newcommand{\Spec}{\operatorname{Spec}}
\newcommand{\pd}{\operatorname{pd}}
\newcommand{\Tor}{\operatorname{Tor}}
\newcommand{\Ext}{\operatorname{Ext}}
\newcommand{\Hom}{\operatorname{Hom}}
\newcommand{\Ann}{\operatorname{Ann}}

\newcommand{\EC}{\mathcal E}
\newcommand{\B}{\mathcal B}
\newcommand{\Tors}{\mathcal T_{\rm reg}}

\title{Baer Splitting Beyond $\tau_q$-Semisimplicity}
\medskip
\author{Xiaolei Zhang$^{a}$,  Guocheng Dai$^{b,\dag}$\\
	\small $a.$ School of Mathematics and Statistics, Tianshui Normal University,
	Tianshui 741001, China\\	
	\small $b.$ School of Mathematical Sciences, Sichuan Normal University,
	Chengdu 610066, China\\
	\small $\dag$\ Corresponding author. E-mail addresses: (Guocheng Dai)  465028571@qq.com}

\begin{document}
\maketitle
\begin{abstract}
For a commutative ring $R$, we call an $R$-module $B$ Baer if
$\Ext_R^1(B,T)=0$ for torsion $R$-module $T$.  It is known that every Baer module is
projective when $R$ is $\tau_q$-semisimple, i.e., rings whose total rings of quotients are semisimple. And it was conjectured that the
converse characterizes $\tau_q$-semisimple rings.  We show that the converse
fails in two substantially different ways.

First, the failure is already systematic in the non-reduced Noetherian case.
If $D$ is a Dedekind domain which is not a field and
$R=D[\varepsilon]/(\varepsilon^2)$, then every Baer $R$-module is projective,
although $R$ is not reduced, and thus is not $\tau_q$-semisimple.

Second, the converse fails even among reduced rings, and in fact among reduced
coherent B\'ezout rings.  The main tool is a countable-ring criterion. For a countable Noetherian domain $D$, let $\EC(D)$ be the ring of eventually
constant sequences over $D$.  If $D$ is not a field, then every Baer
$\EC(D)$-module is projective, whereas $\EC(D)$ is reduced and has countably
infinitely many minimal prime ideals, and thus is not $\tau_q$-semisimple.
\end{abstract}


\maketitle

\section{Introduction}

Let $R$ be a commutative ring with identity and let $\Reg(R)$ denote the
multiplicative set of regular elements of $R$.  An $R$-module $T$ is called
\emph{regular-torsion} if for every $t\in T$ there exists $s\in\Reg(R)$ with
$st=0$.  Following the extension of the classical terminology from domains,
an $R$-module $B$ will be called a \emph{Baer module} if
\[
    \Ext_R^1(B,T)=0
\]
for every regular-torsion module $T$.

For domains, Kaplansky's Baer splitting problem asked whether every Baer module
is projective.  The problem was solved affirmatively by Angeleri H\"ugel,
Bazzoni and Herbera \cite{ABH}, after the structural reduction of
Eklof--Fuchs--Shelah \cite{EFS}.  A recent extension to rings with zero
divisors proved that every Baer module is projective over a
$\tau_q$-semisimple ring \cite{ZKBaer}.  The same work proposed the converse:
\begin{equation}\label{eq:conjecture}
  \text{every Baer $R$-module is projective}
  \quad\Longrightarrow\quad
  R\text{ is $\tau_q$-semisimple}.
\end{equation}
The relevance of reducedness is immediate from the characterization
\begin{equation}\label{eq:tauq-char}
 R\text{ is $\tau_q$-semisimple}
 \iff
 R\text{ is reduced and }\Min(R)\text{ is finite},
\end{equation}
proved in \cite{Ztauq}; equivalently, the total quotient ring is semisimple.
Accordingly, there are two logically distinct ways in which the proposed
converse might fail: nilpotents may already destroy the necessity of
$\tau_q$-semisimplicity, or the converse may fail even after reducedness is
imposed.

The first purpose of the present paper is to make the non-reduced phenomenon
systematic.  We prove that if $D$ is any Dedekind domain which is not a field,
then
\[
   R_D=D[\varepsilon]/(\varepsilon^2)
\]
is a Noetherian non-reduced ring over which every Baer module is projective.
When $D$ is a principal ideal domain, every Baer $R_D$-module is even free.
The key point is that regular-torsion over $R_D$ is exactly torsion over $D$.
Thus a Baer $R_D$-module is projective after restriction of scalars to $D$.
The Baer condition then detects the homology of the square-zero endomorphism
induced by $\varepsilon$, forcing
$\ker(\varepsilon)=\varepsilon B$; the hereditary property of a Dedekind
domain completes the lifting from $D$-projectivity to $R_D$-projectivity.
This argument also clarifies why the original example
$\mathbb Z[\varepsilon]/(\varepsilon^2)$ works and shows that it is not an
isolated accident.  We further prove finite-product stability, producing
non-reduced Noetherian counterexamples with any prescribed positive finite
number of minimal primes.

The second and stronger purpose is to show that reducedness does not rescue
the conjecture.  Our basic reduced counterexample is
\begin{equation}\label{eq:RZ}
 R_{\mathbb Z}
 =\left\{(a_n)_{n\ge 1}\in\mathbb Z^{\mathbb N}:
          (a_n)\text{ is eventually constant}\right\}.
\end{equation}
It is reduced and has infinitely many minimal primes, but every Baer
$R_{\mathbb Z}$-module is projective.  More strongly, $R_{\mathbb Z}$ is a
coherent B\'ezout ring.  Hence adding reducedness, coherence, or B\'ezoutness
to \eqref{eq:conjecture} does not make the converse true.

The proof of the reduced example is not a formal consequence of the
$\tau_q$-semisimple argument.  Indeed, the total quotient ring of
$R_{\mathbb Z}$ is a non-Artinian von Neumann regular ring, so
$R_{\mathbb Z}$ is not subperfect.  Consequently, the subperfect tight-system
theorem used in \cite{ZKBaer} is unavailable.  We replace it with a
countability argument: the Gruson--Jensen cardinality bound gives projective
dimension at most one for every flat module over a countable ring
\cite{GrusonJensen}; countable purification then provides the tight systems
needed for the Eklof--Fuchs--Shelah machinery.

The paper is organized as follows.  In \Cref{sec:prelim} we first analyze the
non-reduced case, generalizing the dual-number example from $\mathbb Z$ to
Dedekind domains and to finite products.  In \Cref{sec:criterion} we establish
a general Baer splitting criterion for countable rings.  In \Cref{sec:EC} we
apply it to eventually constant sequence rings over countable Noetherian
domains.  We also give consequences for $\tau_q$-theory and show that our main
reduced example is $\tau_q$-VN regular although it is not
$\tau_q$-semisimple.  We finish with several questions suggested by the proof.

\section{A Counterexample in Non-Reduced Rings}\label{sec:prelim}

Throughout the paper all rings are commutative with identity and all modules
are unital.  Put $S=\Reg(R)$ and denote by $\Tors$ the class of
regular-torsion $R$-modules.  Thus
\[
 \B={}^\perp\Tors
 :=\{B:\Ext_R^1(B,T)=0\text{ for all }T\in\Tors\}
\]
is the class of Baer modules.

A ring $R$ is called $\tau_q$-semisimple provided that its total quotient ring
$Q(R):=R_{\Reg(R)}$ is semisimple.  By \cite{Ztauq},
\begin{equation}\label{eq:tauq-char-prelim}
 R\text{ is $\tau_q$-semisimple}
 \iff
 R\text{ is reduced and }\Min(R)\text{ is finite}.
\end{equation}
Moreover, \cite{ZKBaer} proves that every Baer module is projective over a
$\tau_q$-semisimple ring and proposes the converse
\begin{equation}\label{eq:conjecture-prelim}
 \text{every Baer $R$-module is projective}
 \quad\Longrightarrow\quad
 R\text{ is $\tau_q$-semisimple}.
\end{equation}
We begin by showing that failure of this converse in the presence of
nilpotents is not confined to a single example.

\begin{lemma}\label{lem:torsion-injective-detects}
Let $D$ be a domain which is not a field.  For every nonzero $D$-module $M$
there exist a torsion injective $D$-module $E$ and a nonzero homomorphism
$M\to E$.
\end{lemma}

\begin{proof}
Choose $0\neq m\in M$ and write $I=\Ann_D(m)$.  If $I\neq0$, then
$Dm\cong D/I$ is a nonzero torsion module.  Let $E=E_D(D/I)$ be its injective
hull.  The module $E$ is torsion: indeed, if $x\in E$ were non-torsion, then
$Dx\cap D/I\neq0$ by essentiality, so for some $0\neq r\in D$ the element
$rx$ would be a nonzero torsion element; then $0\neq s\in D$ with
$srx=0$ would imply $x$ torsion, a contradiction.  The inclusion
$Dm\hookrightarrow E$ extends to a homomorphism $M\to E$, which is nonzero.

If $I=0$, choose a nonzero nonunit $a\in D$.  The map
$Dm\cong D\to D/aD$ sending $m$ to $1+aD$ is nonzero.  Composing with the
inclusion into the injective hull $E_D(D/aD)$ and extending to $M$ again gives
a nonzero map from $M$ to a torsion injective module.
\end{proof}

\begin{theorem}\label{thm:Dedekinddual}
Let $D$ be a Dedekind domain which is not a field and put
\[
 R=D[\varepsilon]/(\varepsilon^2).
\]
Then $R$ is Noetherian and non-reduced, and every Baer $R$-module is
projective.  If $D$ is a principal ideal domain, then every Baer $R$-module is
free.
\end{theorem}

\begin{proof}
The ring $R$ is module-finite over the Noetherian ring $D$, hence Noetherian,
and $0\neq\varepsilon$ with $\varepsilon^2=0$, so it is not reduced.
Exactly as in \Cref{prop:Zdual}, an element $a+b\varepsilon$ is regular if and
only if $a\neq0$.  Therefore an $R$-module is regular-torsion if and only if
its underlying $D$-module is torsion.

Let $B$ be Baer over $R$.  For every torsion $D$-module $U$, the coinduced
module $\Hom_D(R,U)$ is torsion as a $D$-module because
$R\cong D^2$ over $D$; hence it is regular-torsion over $R$.  Derived
adjunction yields
\[
 \Ext_D^1(B,U)
 \cong
 \Ext_R^1\!\left(B,\Hom_D(R,U)\right)=0.
\]
Thus $B$, viewed as a $D$-module, is Baer.  By the solution of the Baer
splitting problem for domains \cite{ABH}, $B_D$ is projective.

Let again $d:B\to B$ be multiplication by $\varepsilon$.  Suppose that
$H=\ker d/\operatorname{Im}d$ is nonzero.  By
\Cref{lem:torsion-injective-detects}, there exist a torsion injective
$D$-module $E$ and a nonzero map $f:H\to E$.  As before, the induced map on
$\ker d$ extends to a $D$-homomorphism $\alpha:B\to E$ with
$\alpha d=0$.  Turn $E$ into an $R$-module by $\varepsilon E=0$ and define
an $R$-module structure on $E\oplus B$ by
\[
 \varepsilon(e,b)=(\alpha(b),d(b)).
\]
Then
\[
 0\longrightarrow E\longrightarrow E\oplus B\longrightarrow B
 \longrightarrow0
\]
is an exact sequence of $R$-modules.  The module $E$ is regular-torsion, so
the sequence splits because $B$ is Baer.  Writing a splitting as
$s(b)=(h(b),b)$ gives $h(db)=\alpha(b)$, and hence $\alpha$ vanishes on
$\ker d$, contrary to the choice of $f$.  Therefore
\[
 \ker d=\operatorname{Im}d=:A.
\]

Since $D$ is Dedekind, it is hereditary; hence the $D$-submodule $A$ of the
projective $D$-module $B$ is projective.  Consequently the exact sequence
\[
 0\longrightarrow A\longrightarrow B\xrightarrow{d}A\longrightarrow0
\]
splits over $D$.  Choose a $D$-linear section $\sigma:A\to B$ of $d$.
Exactly as before,
\[
 R\otimes_DA\longrightarrow B,
 \qquad
 1\otimes a\longmapsto\sigma(a),\qquad
 \varepsilon\otimes a\longmapsto a,
\]
is an $R$-linear isomorphism.  Since $A$ is projective over $D$,
$R\otimes_DA$ is projective over $R$ (tensor a splitting of $A$ from a free
$D$-module with the free $D$-algebra $R$).  Hence $B$ is projective over $R$.
If $D$ is a PID, then $A$ is free, so $B\cong R\otimes_DA$ is free over $R$.
\end{proof}

\begin{remark}\label{rem:totalquotientdual}
Let $K$ be the quotient field of $D$.  For the ring in
\Cref{thm:Dedekinddual}, localization at the regular elements gives
\[
 Q(R)\cong K[\varepsilon]/(\varepsilon^2).
\]
In particular, $Q(R)$ is Artinian but not semisimple.  Moreover
$\sqrt{0}=(\varepsilon)$ is the unique minimal prime of $R$.  Thus these
examples fail $\tau_q$-semisimplicity solely because of nilpotence, not because
of an infinite minimal spectrum.
\end{remark}

\begin{corollary}\label{prop:Zdual}
	Let $R=\mathbb Z[\varepsilon]/(\varepsilon^2)$.  Then the Noetherian ring $R$
	is not reduced, but every Baer $R$-module is a free $R$-module.  Hence $R$ is
	a counterexample to \eqref{eq:conjecture-prelim}.
\end{corollary}

The property is compatible with finite decompositions, which enlarges the
non-reduced family considerably.

\begin{proposition}\label{prop:finiteproductBaer}
Let $R_1,\ldots,R_n$ be commutative rings such that every Baer $R_i$-module is
projective, and put $R=\prod_{i=1}^nR_i$.  Then every Baer $R$-module is
projective.
\end{proposition}

\begin{proof}
Let $e_i\in R$ be the standard central idempotents.  Every $R$-module has the
canonical decomposition $B\cong\prod_i B_i$, where $B_i=e_iB$ is an
$R_i$-module; for a finite product the product and direct sum coincide.
An element $(s_1,\ldots,s_n)\in R$ is regular if and only if every $s_i$ is
regular in $R_i$.  It follows that an $R$-module $T\cong\prod_iT_i$ is
regular-torsion if and only if every $T_i$ is regular-torsion over $R_i$.
Also
\[
 \Ext_R^1(B,T)\cong\prod_{i=1}^n\Ext_{R_i}^1(B_i,T_i).
\]
If $B$ is Baer over $R$, fix $i$ and take a regular-torsion $R_i$-module
$T_i$, with all other components equal to zero.  Then the displayed
isomorphism shows that $B_i$ is Baer over $R_i$.  Hence every $B_i$ is
projective, and therefore $B\cong\prod_iB_i$ is projective over $R$.
\end{proof}

\begin{corollary}\label{cor:manyminimalnonreduced}
For every integer $n\ge1$ there exists a commutative non-reduced Noetherian
ring with exactly $n$ minimal prime ideals such that every Baer module is
projective.  One may take
\[
 R=\prod_{i=1}^n D_i[\varepsilon_i]/(\varepsilon_i^2),
\]
where each $D_i$ is a Dedekind domain which is not a field.  If all $D_i$ are
PIDs, then every Baer $R$-module is a finite product of free modules over the
factors.
\end{corollary}

\begin{proof}
Apply \Cref{thm:Dedekinddual,prop:finiteproductBaer}.  Each factor has the
unique minimal prime $(\varepsilon_i)$, while the minimal primes of a finite
product are obtained by choosing one factor and taking its minimal prime there
and the whole ring in all other components.  Hence there are exactly $n$
minimal primes.
\end{proof}

This family shows that the reducedness hypothesis in
\eqref{eq:tauq-char-prelim} is genuinely independent of the Baer splitting
property.  The substantially stronger question is whether the conjecture
might become true after reducedness is imposed.  The remaining sections show
that the answer is still negative.

\section{A countable-ring criterion for Baer splitting}\label{sec:criterion}

In the following sections, we are devoted to give a reduced ring $R$	such that every Baer $R$-module is projective, but
$R$ is not $\tau_q$-semisimple; equivalently in our reduced examples, it has infinitely many minimal primes.

We shall repeatedly use the following standard character-duality identity.  If
$A$ and $M$ are $R$-modules and $(-)^+=\Hom_{\mathbb Z}(-,\mathbb Q/\mathbb
Z)$, then
\begin{equation}\label{eq:char-dual}
  \Tor_1^R(A,M)^+\cong \Ext_R^1(M,A^+).
\end{equation}
In particular, if $J$ is an ideal containing a regular element $s$, then
$(R/J)^+$ is regular-torsion, since it is annihilated by $s$.

We also use three results on countable direct limits from
Angeleri H\"ugel--Bazzoni--Herbera \cite[Section~2]{ABH}.  We record them in a
form adapted to our application.

\begin{lemma}[Mittag--Leffler package]\label{lem:MLpackage}
Let
\[
 F_1\xrightarrow{f_1}F_2\xrightarrow{f_2}\cdots
\]
be a countable direct system of finitely generated free $R$-modules and put
$B=\varinjlim F_n$.
\begin{enumerate}[label=\textup{(\arabic*)}]
\item For every $R$-module $M$,
\[
 \Ext_R^1(B,M^{(\mathbb N)})=0
\]
if and only if the inverse tower
\[
 \bigl(\Hom_R(F_n,M),\Hom_R(f_n,M)\bigr)_{n\in\mathbb N}
\]
satisfies the Mittag--Leffler condition.
\item For a family $(M_i)_{i\in I}$, the tower obtained with
$\bigoplus_iM_i$ is Mittag--Leffler if and only if the tower obtained with
$\prod_iM_i$ is Mittag--Leffler.
\item If $N\subseteq M$ is pure and the tower obtained with $M$ is
Mittag--Leffler, then the tower obtained with $N$ is Mittag--Leffler.
\item $B$ is projective if and only if the tower obtained with $R$ is
Mittag--Leffler.
\end{enumerate}
\end{lemma}

Parts (1) and (4) are \cite[Proposition~2.5]{ABH}, part (2) is
\cite[Corollary~2.6]{ABH}, and part (3) is \cite[Proposition~2.9]{ABH}.
We shall also use that a countably presented flat module admits such a direct
system; see \cite[Proposition~3.1]{ABH}.

The following consequence of the Gruson--Jensen cardinality theorem is
essential for the uncountable reduction.

\begin{theorem}[Gruson--Jensen]\label{thm:GJ}
If $R$ is an at most countable ring, then every flat $R$-module has projective
dimension at most one.
\end{theorem}

This is the countable case of the cardinality bound on the pure global
dimension of a ring; for a flat module, pure-projective dimension agrees with
projective dimension.  See \cite{GrusonJensen}.  We shall also use the
standard fact that over a countable ring every countably generated module is
countably presented: if $R^{(\mathbb N)}\twoheadrightarrow M$, then
$R^{(\mathbb N)}$ is a countable set, hence its kernel is countable and
therefore countably generated.

We isolate two ring-theoretic conditions.

\begin{definition}\label{def:RCPE}
Let $R$ be a commutative ring.
\begin{enumerate}[label=\textup{(\mathrm{RC})}]
\item[(RC)] For every finitely generated ideal $I$ of $R$, there exists an ideal
$K$ such that
\[
 I\cap K=0,\qquad J:=I\oplus K\text{ contains a regular element.}
\]
\item[(PE)] The diagonal homomorphism
\begin{equation}\label{eq:mu}
 \mu_R:R\longrightarrow \prod_{s\in\Reg(R)}R/sR,
 \qquad x\longmapsto (x+sR)_s,
\end{equation}
is a pure monomorphism.
\end{enumerate}
\end{definition}

Condition (RC) is designed precisely to recover flatness of Baer modules
without assuming $\tau_q$-semisimplicity.

\begin{proposition}\label{prop:flat}
If $R$ satisfies \textup{(RC)}, then every Baer $R$-module is flat.
\end{proposition}

\begin{proof}
Let $B$ be Baer and let $I$ be a finitely generated ideal.  Choose $K$ and
$J=I\oplus K$ as in (RC), and choose $s\in J\cap\Reg(R)$.  Since $s(R/J)=0$,
the character module $(R/J)^+$ belongs to $\Tors$.  Hence
\[
 \Ext_R^1(B,(R/J)^+)=0.
\]
By \eqref{eq:char-dual}, $\Tor_1^R(R/J,B)=0$.  Thus the multiplication map
\[
 J\otimes_RB\longrightarrow B
\]
is injective.  Since $J=I\oplus K$, the module $I\otimes_RB$ is a direct
summand of $J\otimes_RB$, and the restriction
$I\otimes_RB\to B$ is injective.  Equivalently,
$\Tor_1^R(R/I,B)=0$.  This holds for every finitely generated ideal $I$, so
$B$ is flat.
\end{proof}

For modules of projective dimension at most one, the Baer condition can be
tested against direct sums of cyclic regular-torsion modules.

\begin{lemma}\label{lem:test}
Put
\[
 G_R=\bigoplus_{s\in\Reg(R)}R/sR.
\]
Let $X$ be an $R$-module with $\pd_RX\le 1$.
\begin{enumerate}[label=\textup{(\arabic*)}]
\item $X$ is Baer if and only if
$\Ext_R^1(X,T)=0$ for every direct sum $T$ of modules of the form $R/sR$,
$s\in\Reg(R)$.
\item Suppose in addition that $R$ is countable, $\kappa$ is an infinite
cardinal, and $X$ is $\kappa$-generated.  If
\[
 \Ext_R^1(X,G_R^{(\kappa)})=0,
\]
then $X$ is Baer.
\end{enumerate}
\end{lemma}

\begin{proof}
For (1), let $T\in\Tors$.  For each $t\in T$, choose $s_t\in\Reg(R)$ with
$s_tt=0$.  The maps $R/s_tR\to T$, $1+s_tR\mapsto t$, combine to an
epimorphism
\[
 \bigoplus_{t\in T}R/s_tR\twoheadrightarrow T.
\]
If $K$ is its kernel, the long exact Ext sequence and $\Ext_R^2(X,K)=0$ show
that vanishing on the direct sum implies $\Ext_R^1(X,T)=0$.  The converse is
immediate.

For (2), take an exact sequence
\[
 0\longrightarrow H\longrightarrow F\longrightarrow X\longrightarrow0
\]
with $F$ free of rank $\kappa$.  Since $\pd_RX\le1$, $H$ is projective.  As
$R$ is countable and $\kappa$ is infinite, $|F|=\kappa$, hence $H$ is at most
$\kappa$-generated.  Let $T$ be an arbitrary direct sum of cyclic modules
$R/sR$ and let $h:H\to T$.  The image of a generating set of $H$ is contained
in a direct summand $T'\subseteq T$ having at most $\kappa$ cyclic summands.
Such a $T'$ is isomorphic to a direct summand of $G_R^{(\kappa)}$; hence
$\Ext_R^1(X,T')=0$.  Therefore $h$ extends to $F$.  Thus
$\Ext_R^1(X,T)=0$, and (1) applies.
\end{proof}

We shall use the following standard observation.

\begin{lemma}\label{lem:tight-baer}
Let $B$ be Baer and let $A\subseteq B$.  If $\pd_R(B/A)\le1$, then $A$ is
Baer.
\end{lemma}

\begin{proof}
For $T\in\Tors$, the exact sequence $0\to A\to B\to B/A\to0$ gives
\[
 0=\Ext_R^1(B,T)\longrightarrow \Ext_R^1(A,T)
 \longrightarrow \Ext_R^2(B/A,T)=0.
\]
\end{proof}

We now adapt the reduction theorem of Eklof--Fuchs--Shelah.  The point is that
their stationary-set lemma and their singular compactness theorem are stated
for an arbitrary fixed ring; the domain hypothesis in their application is
used to obtain a suitable tight system and countably presented small Baer
modules.  Under our hypotheses these inputs come from countability and
flatness.

\begin{theorem}[Countable reduction]\label{thm:reduction}
Let $R$ be a countable commutative ring and assume that every Baer $R$-module
is flat.  Then an $R$-module $B$ is Baer if and only if it admits a continuous
well-ordered filtration
\[
 0=B_0\subseteq B_1\subseteq\cdots\subseteq B_\alpha
 \subseteq\cdots\subseteq B_\lambda=B
\]
such that every factor $B_{\alpha+1}/B_\alpha$ is a countably generated Baer
module.
\end{theorem}

\begin{proof}
The ``if'' direction is the Eklof lemma, applied to $\Ext_R^1(-,T)$ for each
$T\in\Tors$.

For the converse, let $B$ be Baer.  By assumption $B$ is flat, and by
\Cref{thm:GJ}, $\pd_RB\le1$.  We first record the tight system used in both
cardinal cases.  Let $\mathscr P$ be the family of pure submodules of $B$.  It
contains $0$ and $B$ and is closed under unions of chains.  If
$A\subseteq A'$ lie in $\mathscr P$, then $A$ is pure in $A'$, and the pure
exact sequence
\[
 0\longrightarrow A\longrightarrow A'\longrightarrow A'/A\longrightarrow0
\]
shows that $A'/A$ is flat, since $A'$ is flat.  Hence
$\pd_R(A'/A)\le1$ by \Cref{thm:GJ}.

We shall also use the following purification fact.  If $C$ is a subset of an
$R$-module $M$, then, because $R$ is countable, $C$ is contained in a pure
submodule generated by at most $\max\{|C|,\aleph_0\}$ elements.  Indeed, start
with the submodule generated by $C$, adjoin solutions in $M$ of all finite
systems of linear equations with constants in the submodule, and iterate this
construction $\omega$ times.  At every stage at most
$\max\{|C|,\aleph_0\}$ elements are added.  Applied in $B/A$ and pulled back
to $B$, this implies that if $A\in\mathscr P$ and $C\subseteq B$ is
countable, then there is $A'\in\mathscr P$ containing $A\cup C$ such that
$A'/A$ is countably generated.  Thus $\mathscr P$ is a tight system in the
sense used in \cite[Section~3]{EFS}.

We induct on the least infinite cardinal $\kappa$ for which $B$ is
$\kappa$-generated.  The case $\kappa=\aleph_0$ is immediate.

\smallskip
\noindent\emph{Case 1: $\kappa$ is regular uncountable.}
Fix a generating family $(b_\alpha)_{\alpha<\kappa}$ of $B$.  By recursion,
using the purification fact in the quotient at each successor stage, construct
a continuous chain
\[
 0=M_0\subseteq M_1\subseteq\cdots\subseteq M_\alpha\subseteq\cdots
 \qquad(\alpha<\kappa)
\]
of members of $\mathscr P$ such that $b_\alpha\in M_{\alpha+1}$,
$M_{\alpha+1}/M_\alpha$ is countably generated, and
$B=\bigcup_{\alpha<\kappa}M_\alpha$.  Since $\kappa$ is regular, each
$M_\alpha$ is generated by fewer than $\kappa$ elements.  Moreover every
successive quotient $M_{\alpha+1}/M_\alpha$ is flat, hence has projective
dimension at most one by \Cref{thm:GJ}.  Thus the chain satisfies the
hypotheses of \cite[Lemma~9]{EFS}.

Put
\[
 G_R=\bigoplus_{s\in\Reg(R)}R/sR,
 \qquad G_\alpha=G_R^{(\aleph_0)}\quad(\alpha<\kappa),
\]
and set
\[
 H_\beta=\bigoplus_{\alpha<\beta}G_\alpha,
 \qquad H=\bigoplus_{\alpha<\kappa}G_\alpha.
\]
Every element of $H$ has finite support and is annihilated by a finite product
of regular elements; hence $H\in\Tors$ and $\Ext_R^1(B,H)=0$.  By the
contrapositive of \cite[Lemma~9]{EFS}, the set
\[
 E=\left\{\alpha<\kappa:\text{ for some }\beta>\alpha,
 \ \Ext_R^1(M_\beta/M_\alpha,H_\beta/H_\alpha)\ne0\right\}
\]
is nonstationary.  Choose a club $C\subseteq\kappa$ disjoint from $E$ and
thin the chain to the indices in $C$.

Let $\alpha<\beta$ be two consecutive points of the thinned chain.  Since
$\alpha\notin E$,
\begin{equation}\label{eq:club-vanishing}
 \Ext_R^1(M_\beta/M_\alpha,H_\beta/H_\alpha)=0.
\end{equation}
Set
\[
 \mu=\max\{\aleph_0,|\beta\setminus\alpha|\}.
\]
Because $\beta<\kappa$ and $\kappa$ is regular, $\mu<\kappa$.  The quotient
$M_\beta/M_\alpha$ is $\mu$-generated: it is generated by lifts of countable
generating sets for the quotients $M_{\gamma+1}/M_\gamma$ with
$\alpha\le\gamma<\beta$.  Also
\[
 H_\beta/H_\alpha\cong G_R^{(\mu)}.
\]
Finally, $M_\alpha$ is pure in $M_\beta$, so $M_\beta/M_\alpha$ is flat and
therefore has projective dimension at most one.  Consequently
\eqref{eq:club-vanishing} and \Cref{lem:test}(2) imply that
$M_\beta/M_\alpha$ is Baer.  Its least infinite number of generators is less
than $\kappa$, so the induction hypothesis refines it by a filtration with
countably generated Baer factors.  Refining every successive quotient of the
club-thinned chain and concatenating the refinements gives the required
filtration of $B$.

\smallskip
\noindent\emph{Case 2: $\kappa$ is singular.}
Let $\mathcal C$ be a set of representatives of the isomorphism classes of
countably generated Baer $R$-modules.  Such a set exists; moreover every
member of $\mathcal C$ is countably presented because $R$ is countable.  Let
\[
 \mathcal F=\operatorname{Filt}(\mathcal C)
\]
be the class of modules admitting a filtration by members of $\mathcal C$.
The construction preceding \cite[Theorem~11]{EFS}, together with that theorem,
shows that the filtration class determined by a set of countably presented
modules satisfies the abstract hypotheses (a)--(e) of the singular
compactness theorem \cite[Theorem~10]{EFS}.  We verify explicitly the three
families required by that theorem.

For each infinite cardinal $\theta<\kappa$, put
\[
 \mathscr W_\theta=
 \{A\in\mathscr P: A\text{ is generated by at most }\theta\text{ elements}\}.
\]
First, $\mathscr W_\theta\subseteq\mathcal F$.  Indeed, if
$A\in\mathscr W_\theta$, then $B/A$ is flat because $A$ is pure in $B$; hence
$\pd_R(B/A)\le1$ by \Cref{thm:GJ}.  By \Cref{lem:tight-baer}, $A$ is Baer.
Since $A$ is generated by at most $\theta<\kappa$ elements, the induction
hypothesis yields $A\in\mathcal F$.

Second, $\mathscr W_\theta$ is closed under unions of increasing chains of
length $<\theta$.  Such a union is pure in $B$, and an infinite cardinal
$\theta$ times fewer than $\theta$ many generating sets of cardinality at
most $\theta$ still has cardinality at most $\theta$.

Third, $\mathscr W_\theta$ is cofinal for subsets of cardinality $<\theta$.
If $X\subseteq B$ and $|X|<\theta$, the purification construction above gives
a pure submodule $A$ containing $X$ and generated by at most
$\max\{|X|,\aleph_0\}\le\theta$ elements.  Hence $A\in\mathscr W_\theta$.

Thus conditions (1')--(3') of \cite[Theorem~10]{EFS} hold for every infinite
$\theta<\kappa$.  Singular compactness gives $B\in\mathcal F$, which is
precisely the desired filtration.  This completes the induction.
\end{proof}

We can now prove the main criterion.

\begin{theorem}[Countable-ring Baer splitting criterion]\label{thm:criterion}
Let $R$ be a countable commutative ring satisfying \textup{(RC)} and
\textup{(PE)} from \Cref{def:RCPE}.  Then every Baer $R$-module is projective.
\end{theorem}

\begin{proof}
By \Cref{prop:flat}, every Baer module is flat.  Hence
\Cref{thm:reduction} reduces the problem to countably generated Baer modules.
Let $B$ be countably generated and Baer.  Because $R$ is countable, $B$ is
countably presented; it is flat by \Cref{prop:flat}.  By
\cite[Proposition~3.1]{ABH}, write
\[
 B\cong\varinjlim
 \left(F_1\xrightarrow{f_1}F_2\xrightarrow{f_2}\cdots\right)
\]
with the $F_n$ finitely generated free.

Put $G_R=\bigoplus_{s\in\Reg(R)}R/sR$.  The module
$G_R^{(\mathbb N)}$ is regular-torsion, so
\[
 \Ext_R^1(B,G_R^{(\mathbb N)})=0.
\]
By \Cref{lem:MLpackage}(1), the tower obtained from $G_R$ is
Mittag--Leffler.  By \Cref{lem:MLpackage}(2), so is the tower obtained from
\[
 P_R:=\prod_{s\in\Reg(R)}R/sR.
\]
Condition (PE) says that $R\subseteq P_R$ is pure.  Therefore
\Cref{lem:MLpackage}(3) shows that the tower obtained from $R$ is
Mittag--Leffler.  Finally, \Cref{lem:MLpackage}(4) gives that $B$ is
projective.

Thus every countably generated Baer module is projective.  By
\Cref{thm:reduction}, an arbitrary Baer module has a filtration with such
factors.  Every successor extension splits, so a transfinite induction shows
that the whole module is a direct sum of projective factors and hence is
projective.
\end{proof}

\begin{remark}\label{rem:subperfect}
No subperfectness assumption occurs in \Cref{thm:criterion}.  This is important
for the examples below: their total quotient rings are von Neumann regular but
not Artinian, hence are not perfect.
\end{remark}

\section{Eventually constant sequence rings}\label{sec:EC}

Let $D$ be a domain.  Define
\begin{equation}\label{eq:ED}
 \EC(D)=\left\{(a_n)_{n\ge1}\in D^{\mathbb N}:
            (a_n)\text{ is eventually constant}\right\}.
\end{equation}
Let $\mathbf 1=(1,1,\ldots)$ and let $e_n$ be the sequence having $1$ in the
$n$th coordinate and $0$ elsewhere.  Then
\[
 \EC(D)=D\mathbf1+\bigoplus_{n\ge1}De_n.
\]
For $N\ge1$, put
\[
 e_\infty^{(N)}=\mathbf1-e_1-\cdots-e_N,
 \qquad
 R_N=De_1\oplus\cdots\oplus De_N\oplus De_\infty^{(N)}.
\]
Then $R_N\cong D^{N+1}$ and $\EC(D)=\bigcup_NR_N$.

\begin{proposition}\label{prop:structure}
Let $R=\EC(D)$.
\begin{enumerate}[label=\textup{(\arabic*)}]
\item $R$ is reduced and
\[
 \Reg(R)=\{(a_n)\in R:a_n\ne0\text{ for all }n\}.
\]
\item Every prime ideal of $R$ is uniquely of one of the forms
\[
 \mathfrak p_{n,\mathfrak q}
   =\{a\in R:a_n\in\mathfrak q\},
 \qquad
 \mathfrak p_{\infty,\mathfrak q}
   =\{a\in R:a_\infty\in\mathfrak q\},
 \qquad \mathfrak q\in\Spec(D),
\]
where $a_\infty$ denotes the eventual value of $a$.  Consequently
$\dim R=\dim D$.  In particular, the minimal prime ideals are exactly
\[
 \mathfrak p_n:=\mathfrak p_{n,(0)}=\{a\in R:a_n=0\}\quad(n\ge1),
 \qquad
 \mathfrak p_\infty:=\mathfrak p_{\infty,(0)}
   =\bigoplus_{n\ge1}De_n,
\]
and $\Min(R)$ is countably infinite.
\item If $K=\operatorname{Frac}(D)$, then
\[
 T(R)\cong \EC(K).
\]
The ring $\EC(K)$ is von Neumann regular and is not Artinian.
\item With the subspace topology inherited from $\Spec(R)$,
$\Min(R)$ is homeomorphic to the one-point compactification of the discrete
space $\mathbb N$, with $\mathfrak p_\infty$ as the point at infinity.  In
particular, $\Min(R)$ is compact.
\end{enumerate}
\end{proposition}

\begin{proof}
Since $R$ is a subring of the product of copies of the domain $D$, it is
reduced.  An element $a=(a_n)$ is a zero-divisor if one of its coordinates,
say $a_i$, is zero, because $ae_i=0$.  Conversely, if every $a_i\ne0$ and
$ab=0$, then $a_ib_i=0$ for every $i$, hence $b_i=0$ for every $i$.  This
proves (1).

Let $\pi_n:R\to D$, $a\mapsto a_n$, be evaluation at the $n$th coordinate,
and let $\pi_\infty:R\to D$ be the eventual-value map.  Both maps are
surjective.  Let $P\in\Spec(R)$.  If $e_n\notin P$ for some $n$, then
$e_n(1-e_n)=0$ and primality imply $1-e_n\in P$.  Hence
$\ker(\pi_n)=(1-e_n)R\subseteq P$, and therefore
$P=\pi_n^{-1}(\mathfrak q)$ for a unique $\mathfrak q\in\Spec(D)$.  If
$e_n\in P$ for every $n$, then
$\ker(\pi_\infty)=\bigoplus_{n\ge1}De_n\subseteq P$, so
$P=\pi_\infty^{-1}(\mathfrak q)$ for a unique
$\mathfrak q\in\Spec(D)$.  The two alternatives are disjoint and give all
prime ideals.  A chain of primes cannot move from one branch to another, and
within each branch inclusion is exactly inclusion of the corresponding primes
of $D$; hence $\dim R=\dim D$.

Taking $\mathfrak q=(0)$ gives the displayed minimal primes.  Alternatively,
minimality is also seen from
$R_{\mathfrak p_n}\cong K\cong R_{\mathfrak p_\infty}$.  For example,
$e_n\notin\mathfrak p_n$, so $e_n$ becomes $1$ after localization and all
other orthogonal idempotent components vanish; the surviving copy of $D$ is
localized at $D\setminus\{0\}$.  At $\mathfrak p_\infty$ the elements
$1-e_n$ become units and force $e_n=0$ for every $n$.  This proves (2).

Every fraction of eventually constant $D$-sequences is an eventually constant
$K$-sequence.  Conversely, a sequence in $\EC(K)$ has only finitely many
exceptional coordinates, so a common denominator gives a representation
$x/s$ with $x\in R$ and $s\in\Reg(R)$.  Thus $T(R)=\EC(K)$.  Coordinatewise,
for $x=(x_n)\in\EC(K)$ define $y_n=x_n^{-1}$ when $x_n\ne0$ and $y_n=0$
otherwise.  Then $y\in\EC(K)$ and $x=x^2y$, so $\EC(K)$ is von Neumann
regular.  It has the infinite family of nonzero orthogonal idempotents
$(e_n)$, hence it is not Artinian.  This proves (3).

Finally, for $f=(f_n)\in R$, the basic open set $D(f)\cap\Min(R)$ contains
$\mathfrak p_n$ exactly when $f_n\ne0$, and contains $\mathfrak p_\infty$
exactly when the eventual value $f_\infty$ is nonzero.  If $f_\infty=0$, then
$f$ has finite support, so $D(f)\cap\Min(R)$ is finite and avoids
$\mathfrak p_\infty$.  If $f_\infty\ne0$, then $f_n=f_\infty\ne0$ for all
large $n$, so the basic open contains $\mathfrak p_\infty$ and all but
finitely many $\mathfrak p_n$.  Conversely, for every finite set
$F\subseteq\mathbb N$, the element $1-\sum_{n\in F}e_n$ defines the basic
neighborhood
\[
 \{\mathfrak p_\infty\}\cup
 \{\mathfrak p_n:n\notin F\}
\]
of $\mathfrak p_\infty$.  This is precisely the one-point compactification
topology.  Thus (4) follows.
\end{proof}

The first condition of \Cref{thm:criterion} is automatic for $\EC(D)$.

\begin{proposition}\label{prop:RCED}
For every domain $D$, the ring $R=\EC(D)$ satisfies \textup{(RC)}.
\end{proposition}

\begin{proof}
Let $I=(x_1,\ldots,x_m)$ be a finitely generated ideal.  Choose $N$ such that
each $x_j$ belongs to $R_N\cong D^{N+1}$.  For
$i\in\{1,\ldots,N,\infty\}$, let $I_i$ be the corresponding coordinate ideal
of $D$.  For every nonzero $I_i$, choose $0\ne a_i\in I_i$.  Write
\[
 a_i=\sum_{j=1}^m c_{ij}(x_j)_i\qquad(c_{ij}\in D)
\]
for $i=1,\ldots,N$, and use analogous coefficients $c_{\infty j}$ for the
eventual coordinate when $I_\infty\ne0$ (put all these coefficients equal to
zero when the corresponding coordinate ideal is zero).  Define
\[
 r_j=\sum_{i=1}^N c_{ij}e_i+c_{\infty j}e_\infty^{(N)}\in R_N.
\]
Then
\[
 a:=\sum_{j=1}^m r_jx_j
 =\sum_{\substack{1\le i\le N\\ I_i\ne0}}a_ie_i
   +\begin{cases}
      a_\infty e_\infty^{(N)},& I_\infty\ne0,\\
      0,&I_\infty=0,
    \end{cases}
\]
so $a\in I$ and has exactly the prescribed nonzero coordinate components.

Let $e$ be the sum of those idempotent components among
$e_1,\ldots,e_N,e_\infty^{(N)}$ on which $I_i=0$.  Then $eI=0$.  Hence
$I\cap eR=0$: if $z\in I\cap eR$, then $z=ez$, while $ez=0$ because
$z\in I$.  Moreover, $a+e$ has no zero coordinate, so it is regular by
\Cref{prop:structure}(1).  Therefore $I\oplus eR$ is a regular ideal.
\end{proof}

We next establish the pure embedding.  We first isolate a separation lemma.

\begin{lemma}\label{lem:Dseparated}
Let $D$ be a Noetherian domain which is not a field and let $M$ be a finitely
generated $D$-module.  Then
\[
 \bigcap_{0\ne d\in D} dM=0.
\]
\end{lemma}

\begin{proof}
Suppose $0\ne x\in M$.  Choose a maximal ideal $\mathfrak m$ containing
$\Ann_D(x)$.  Then $x/1\ne0$ in $M_{\mathfrak m}$: otherwise some
$s\notin\mathfrak m$ would satisfy $sx=0$, forcing
$s\in\Ann_D(x)\subseteq\mathfrak m$, a contradiction.  Since $D$ is not a
field, $\mathfrak m\ne0$; choose $0\ne a\in\mathfrak m$.  If
$x\in\bigcap_{0\ne d}dM$, then $x\in a^nM$ for every $n\ge1$, so
\[
 x/1\in\bigcap_{n\ge1}a^nM_{\mathfrak m}.
\]
The Krull intersection theorem for the finitely generated module
$M_{\mathfrak m}$ over the Noetherian local ring $D_{\mathfrak m}$ gives
$\bigcap_na^nM_{\mathfrak m}=0$, a contradiction.
\end{proof}

\begin{proposition}\label{prop:PEED}
Let $D$ be a Noetherian domain which is not a field and put $R=\EC(D)$.  Then
$R$ satisfies \textup{(PE)}.
\end{proposition}

\begin{proof}
To test purity it is enough to tensor \eqref{eq:mu} with finitely presented
$R$-modules.  If $M$ is finitely presented, tensoring with $M$ commutes with
products, and therefore
\[
 M\otimes_R\prod_{s\in\Reg(R)}R/sR
 \cong \prod_{s\in\Reg(R)}M/sM.
\]
Thus it is enough to show that, for every finitely presented $R$-module $M$,
the canonical map
\[
 M\longrightarrow \prod_{s\in\Reg(R)}M/sM
\]
is injective.  Its kernel is $\bigcap_{s\in\Reg(R)}sM$.  Since every nonzero
constant $d\in D$ is regular in $R$,
\begin{equation}\label{eq:intersection-bound}
 \bigcap_{s\in\Reg(R)}sM
 \subseteq
 \bigcap_{0\ne d\in D}dM.
\end{equation}

Choose a finite presentation of $M$.  All matrix entries occur in some
$R_N\cong D^{N+1}$.  Hence there is a finitely presented $R_N$-module $L$ such that
\[
 M\cong L\otimes_{R_N}R.
\]
Decompose
\[
 L=L_1\oplus\cdots\oplus L_N\oplus L_\infty
\]
along the $N+1$ central idempotents of $R_N$.  Each $L_i$ and $L_\infty$ is a
finitely generated $D$-module.  Moreover, as a $D$-module,
\[
 e_\infty^{(N)}R
 =De_\infty^{(N)}\oplus\bigoplus_{j>N}De_j,
\]
which is free of countable rank.  Therefore
\[
 M\cong L_1\oplus\cdots\oplus L_N
 \oplus L_\infty
 \oplus\bigoplus_{j>N}L_\infty
\]
as a $D$-module.  In particular, the underlying $D$-module of $M$ is a direct
sum of finitely generated $D$-modules.

For such a direct sum the intersection
$\bigcap_{0\ne d\in D}dM$ is zero: a nonzero element has finite support, and
\Cref{lem:Dseparated} applied to any nonzero component supplies a nonzero
$d$ for which the element does not lie in $dM$.  Thus the right-hand side of
\eqref{eq:intersection-bound} is zero.  Therefore the displayed map is
injective for every finitely presented $M$, which is exactly the purity of
\eqref{eq:mu}.
\end{proof}

We can now state the main family of counterexamples.

\begin{theorem}\label{thm:EDmain}
Let $D$ be a countable Noetherian domain which is not a field, and set
$R=\EC(D)$.  Then every Baer $R$-module is projective.  Nevertheless, $R$ is
reduced and $\Min(R)$ is infinite; hence $R$ is not $\tau_q$-semisimple.
\end{theorem}

\begin{proof}
The ring $R$ is countable.  By \Cref{prop:RCED} it satisfies (RC), and by
\Cref{prop:PEED} it satisfies (PE).  Thus \Cref{thm:criterion} shows that every
Baer $R$-module is projective.  The final assertions follow from
\Cref{prop:structure}(1)--(2) and the characterization
\eqref{eq:tauq-char}.
\end{proof}

When the coefficient domain is principal, the counterexamples have strong
finiteness properties.

\begin{corollary}\label{cor:PID}
Let $D$ be a countable principal ideal domain which is not a field and put
$R=\EC(D)$.  Then every Baer $R$-module is projective, while $R$ is reduced
with infinitely many minimal primes.  Moreover, $R$ is B\'ezout,
semihereditary, and coherent.
\end{corollary}

\begin{proof}
Only the last assertions are new.  Let $I=(x_1,\ldots,x_m)$ and choose $N$ so
that all generators lie in $R_N\cong D^{N+1}$.  Since $D$ is a PID, the
coordinate ideals are generated by elements
$d_1,\ldots,d_N,d_\infty$.  As in the proof of \Cref{prop:RCED}, the element
\[
 d=d_1e_1+\cdots+d_Ne_N+d_\infty e_\infty^{(N)}
\]
belongs to $I$.  For every $x\in I$, coordinatewise division by the $d_i$
produces an eventually constant element $r\in R$ with $x=dr$; on a zero
component both $d_i$ and the corresponding coordinate of $x$ are zero.  Thus
$I=dR$ and $R$ is B\'ezout.  If $e$ is the idempotent supported on the zero
components of $d$, then $\Ann_R(d)=eR$ and multiplication by $d$ identifies
$dR$ with $(1-e)R$.  Hence every finitely generated ideal is projective.
Thus $R$ is semihereditary and therefore coherent.
\end{proof}

The simplest instance is obtained from the integers.

\begin{corollary}\label{cor:Z}
Let
\[
 R=\EC(\mathbb Z)
 =\mathbb Z\mathbf1+\bigoplus_{n\ge1}\mathbb Ze_n.
\]
Then:
\begin{enumerate}[label=\textup{(\arabic*)}]
\item every Baer $R$-module is projective;
\item $R$ is reduced, B\'ezout, coherent and semihereditary;
\item $\Min(R)$ is countably infinite and compact;
\item $T(R)\cong\EC(\mathbb Q)$ is von Neumann regular but not semisimple
Artinian;
\item $R$ is not $\tau_q$-semisimple.
\end{enumerate}
\end{corollary}

\begin{proof}
Apply \Cref{cor:PID} with $D=\mathbb Z$, together with
\Cref{prop:structure}(3)--(4).  The last assertion follows from
\eqref{eq:tauq-char} because $\Min(R)$ is infinite.
\end{proof}

The non-field hypothesis in \Cref{thm:EDmain} is essential for this family.

\begin{proposition}\label{prop:field}
Let $K$ be a field and $R=\EC(K)$.  Then every $R$-module is Baer, but not every
$R$-module is projective.
\end{proposition}

\begin{proof}
By \Cref{prop:structure}(1), a regular element of $R$ has no zero coordinate.
Its coordinatewise inverse is again eventually constant, so every regular
element is a unit.  Hence the only regular-torsion module is $0$, and every
module is Baer.

For an explicit nonprojective Baer module, let
\[
 J=\bigoplus_{n\ge1}Ke_n\subseteq R
\]
and consider $R/J$.  If $R/J$ were projective, then the exact sequence
$0\to J\to R\to R/J\to0$ would split, so $J$ would be a direct-summand ideal
of $R$, hence $J=eR$ for an idempotent $e\in R$.  Every idempotent of $R$ is
an eventually constant $0$--$1$ sequence.  If its eventual value is $0$, then
its support is finite, so $eR$ is supported on a fixed finite set and cannot
equal $J$.  If its eventual value is $1$, then $eR$ contains elements with
nonzero eventual value and again cannot equal $J$.  Thus $R/J$ is not
projective.  Since every $R$-module is Baer, it is the required explicit
nonprojective Baer module.
\end{proof}

Combining the last two results gives a sharp statement inside this natural
sequence-ring family.

\begin{corollary}\label{cor:dichotomy}
Let $D$ be a countable Noetherian domain and $R=\EC(D)$.  Then every Baer
$R$-module is projective if and only if $D$ is not a field.
\end{corollary}

The main motivation was the proposed converse to the Baer splitting theorem
over $\tau_q$-semisimple rings.  We can now settle it decisively.

\begin{theorem}\label{thm:conjecturefalse}
The implication
\[
 \bigl[\text{every Baer $R$-module is projective}\bigr]
 \Longrightarrow
 \bigl[R\text{ is $\tau_q$-semisimple}\bigr]
\]
is false even when $R$ is assumed to be reduced.  It remains false even if
$R$ is additionally assumed coherent, B\'ezout, and semihereditary.
\end{theorem}

\begin{proof}
Use $R=\EC(\mathbb Z)$ and \Cref{cor:Z}.
\end{proof}

There is a further separation phenomenon.  Zhang and Qi proved that a
commutative ring is $\tau_q$-VN regular if and only if it is reduced and its
minimal spectrum is compact \cite{ZQflat}.  Therefore:

\begin{corollary}\label{cor:tauqVN}
The ring $R=\EC(\mathbb Z)$ is $\tau_q$-VN regular but not
$\tau_q$-semisimple, while every Baer $R$-module is projective.
\end{corollary}

\begin{proof}
By \Cref{cor:Z}, $R$ is reduced and $\Min(R)$ is compact, so it is
$\tau_q$-VN regular by \cite{ZQflat}.  It is not $\tau_q$-semisimple because
$\Min(R)$ is infinite, by \eqref{eq:tauq-char}.  The Baer splitting assertion
is \Cref{cor:Z}(1).
\end{proof}

Thus the Baer splitting property is strictly weaker than
$\tau_q$-semisimplicity, even inside the reduced $\tau_q$-VN regular class.
The obstruction is not nilpotence and not noncoherence; rather, the example
shows that an infinite but compact minimal spectrum is compatible with the
Baer splitting property.

\end{document}